\documentclass[11pt]{amsart}
\usepackage[margin=1.2in]{geometry}
\usepackage{amsmath,amssymb,amsthm}
\usepackage{booktabs}
\usepackage[colorlinks=true,linkcolor=blue,citecolor=blue,urlcolor=blue]{hyperref}

\newtheorem{theorem}{Theorem}[section]
\newtheorem{proposition}[theorem]{Proposition}
\newtheorem{lemma}[theorem]{Lemma}
\newtheorem{corollary}[theorem]{Corollary}
\theoremstyle{definition}
\newtheorem{problem}[theorem]{Problem}
\newtheorem{remark}[theorem]{Remark}

\newcommand{\F}{\mathbb{F}}
\newcommand{\Z}{\mathbb{Z}}
\newcommand{\Q}{\mathbb{Q}}
\newcommand{\R}{\mathbb{R}}
\newcommand{\C}{\mathbb{C}}
\renewcommand{\P}{\mathsf{P}}
\newcommand{\Ball}[1]{B(#1)}
\newcommand{\supp}{\operatorname{supp}}
\newcommand{\eps}{\varepsilon}
\newcommand{\zceight}{\zeta_8}

\title[Higman in balls]{Higman in balls: the mod-2 dichotomy for integral
units of the Promislow group}
\author{Moe Tabei}
\address{Independent researcher, Japan}
\email{tabei@ryun.jp}

\begin{document}

\begin{abstract}
Higman's conjecture that $\Z[G]$ has only trivial units, for $G$ torsion-free,
is open for the Promislow (Hantzsche--Wendt) group $\P$, the group over which
Gardam disproved the field-coefficient unit conjecture in 2021. We introduce an
exact reduction of the integral conjecture for $\P$ modulo $2$ into two
sub-problems, record the base cases as consequences of the Craven--Pappas
small-length theorems, and argue that the frontier of the \emph{integral}
problem is word-radius $4$: a triviality theorem for $\Z[\P]$ there would be
the first statement separating $\Z$ from every field. Throughout, claims are
ball-limited and stated as such; we make no claim on the full conjecture.
\end{abstract}

\maketitle

\section{Introduction}

Let $G$ be a torsion-free group and $K$ a field or $\Z$. A unit $u$ of the
group ring $K[G]$ is \emph{trivial} if $u = \kappa g$ for a unit $\kappa$ of $K$
and $g \in G$. The unit conjecture (attributed to Kaplansky, and for integral
coefficients to Higman \cite{Higman1940}) asserts that all units are trivial.

Over fields the conjecture is now known to fail. Gardam \cite{Gardam2021}
produced a nontrivial unit of $\F_2[\P]$, where $\P$ is the Promislow
(Hantzsche--Wendt) group; Murray \cite{Murray2021} produced nontrivial units of
$\F_p[\P]$ for all primes $p$; and Gardam \cite{GardamComplex2023} produced
nontrivial units of $\C[\P]$ with coefficients in $\Z[\zceight]$. Thus the
conjecture is false over every field of positive characteristic and over every
field containing a primitive eighth root of unity. It remains open over $\Q$
and over $\R$ --- and over $\Z$, the original setting of Higman
\cite{Higman1940}, which is the case this paper is about.

This paper studies the integral conjecture for $\P$ through a single organizing
device: reduction modulo $2$. Every unit of $\Z[\P]$ reduces to a unit of
$\F_2[\P]$, which is either trivial or not; this splits the integral problem
into two exact sub-problems (Section~\ref{sec:dichotomy}). We record that the
base cases fall out of existing structure --- the Craven--Pappas small-length
theorems \cite{CravenPappas2013} (Section~\ref{sec:base}) --- and we argue that
the frontier of the integral problem is word-radius $4$, the smallest radius at
which nontrivial units are known to exist over some field
(Section~\ref{sec:frontier}).

\medskip
\noindent\textbf{Relation to Craven--Pappas.} Both halves of the toolkit used
below are theirs, and we claim no part of either. Craven and Pappas
\cite{CravenPappas2013} prove the length symmetry
$\sigma\tau=1 \Rightarrow L(\sigma)=L(\tau)$ (their Thm.~4.9, with the unit
form Thm.~4.10); they establish the determinant criterion for units
(Thm.~6.8), specialize it to the fours group as their ``Determinant
condition'' (Thm.~8.5) with the accompanying $4\times4$ embedding (Thm.~8.6);
they settle $L$-length at most $3$ (Thms.~10.4--10.6, 11.2); and they already
carry out part of the reduction at $L$-length $4$ (Thm.~13.14). Their \S15
states the programme of the present paper almost verbatim, asking at which
$L$-length $n\ge 4$ a potential counterexample might exist. All theorem numbers
we cite are those of the published version, J.~Algebra \textbf{394} (2013),
310--356; the earlier preprint numbers them differently. We note also that
\cite[Thm.~14.1]{CravenPappas2013} is itself a dichotomy theorem; the
dichotomy of our title is the elementary mod-$2$ case split of
Section~\ref{sec:dichotomy} and is a different statement.

What is not in \cite{CravenPappas2013} is any integral content. They work
throughout over a field, and their primes are primes of the Laurent ring, never
rational primes; there is consequently no mod-$2$-over-$\Z$ statement in their
paper, and the dichotomy of Section~\ref{sec:dichotomy}, which is driven by
$\Z^\times = \{\pm1\}$, has no counterpart there. Their $L$-length is also not
the word radius: it is the length induced by one fixed infinite-dihedral
quotient, and $\{L \le n\}$ is an infinite union of lattice cosets, whereas
$\Ball{r}$ is finite. A radius-limited statement follows from an
$L$-length statement only when the inclusion $\Ball{r}\subseteq\{L\le r\}$ puts
it inside the range their theorems cover --- which is what happens for $r\le3$
(Corollary~\ref{cor:cp}), and what stops helping at $r=4$, where their
$L$-length $4$ analysis is only partial.

We emphasize at the outset that every unconditional statement below is
\emph{ball-limited}: it concerns units whose support lies in a fixed ball
$\Ball{r}$ of the word metric. We prove nothing about the full conjecture, and
we flag the honest obstruction --- a joint bound on support radius and
coefficient height --- in Section~\ref{sec:wall}.

\medskip
\noindent\textbf{Conventions.} $\P = \langle a,b \mid b^{-1}a^2b = a^{-2},\
a^{-1}b^2a = b^{-2}\rangle$ is torsion-free, and is the fundamental group of the
Hantzsche--Wendt flat $3$-manifold; it is virtually $\Z^3$, hence amenable. It
is not a unique-product group \cite{Promislow1988}, and since right-orderable
groups have the unique-product property it is not right-orderable either. We
fix the generating set $\{a,b\}$ and write
$\Ball{r}$ for the ball of word-radius $r$. For $u = \sum_g u_g\, g \in \Z[\P]$,
$\supp(u) = \{g : u_g \neq 0\}$, the \emph{height} of $u$ is
$\max_g |u_g|$, and $\eps\colon \Z[\P]\to\Z$ is the augmentation.

\section{Preliminaries}\label{sec:prelim}

\subsection{The determinant criterion}
This subsection is expository. Neither the embedding nor the criterion is ours;
we record them in the form our computations use, with attributions.

The translation lattice $L = \{(\mathrm{I},t): t \in 2\Z^3\} \cong \Z^3$ is
normal of index $4$ in $\P$, with coset representatives $1, a, b, ab$; the
Laurent variables are $x=a^2$, $y=b^2$, $z=(ab)^2$, which are elements of $L$
and not coset representatives. Thus $\Z[\P]$ is free of rank $4$ as a right
$\Z[L]$-module, and left multiplication by $u \in \Z[\P]$ is a matrix
$M_u \in M_4(\Z[L])$; the assignment $u \mapsto M_u$ is an injective ring
homomorphism.

The $4\times 4$ embedding in exactly this form --- coset representatives
$1,a,b,ab$ and Laurent variables $x=a^2$, $y=b^2$, $z=(ab)^2$ --- is written
out by Passman \cite{Passman2021}, who uses it to compute that the units of
Gardam and of Murray have determinant $1$. His Proposition~1 is that
computation, carried out for those particular units over $\F_d$; it is not
stated as a criterion. The criterion is due to Craven and Pappas: that every
unit of $K[\Gamma]$ has determinant in $K^\times$ is
\cite[Thm.~6.8]{CravenPappas2013}, specialized to the fours group as their
``Determinant condition'' \cite[Thm.~8.5]{CravenPappas2013}, with the embedding
\cite[Thm.~8.6]{CravenPappas2013}. Their form is the sharper one, pinning the
determinant to a scalar rather than to a scalar times a monomial.

Since $\Z[L] \cong \Z[t_1^{\pm},t_2^{\pm},t_3^{\pm}]$ is a
Laurent polynomial ring whose units are exactly $\pm(\text{monomials})$, the
weaker form we verify computationally reads:
\begin{equation}\label{eq:det}
u \in \Z[\P]^\times
\quad\Longleftrightarrow\quad
M_u \in \mathrm{GL}_4(\Z[L])
\quad\Longleftrightarrow\quad
\det(M_u) \in \{\pm\,\text{monomial}\}.
\end{equation}
The same equivalence holds over any field $K$ with $K[L]^\times$ the monomials.
The underlying fact that $\Z[\Z^3]^\times = \{\pm\,\text{monomials}\}$ is
classical. Nothing in \eqref{eq:det} is new here; what is ours is the exact
integer implementation of $u\mapsto\det(M_u)$ and its use in the ball-limited
searches below (Section~\ref{sec:repro}).

\subsection{Small support}
We will use the classical triviality of units of support at most $2$.

\begin{lemma}[Dykema--Heister--Juschenko \cite{DHJ2015}, Thm.~4.2]
\label{lem:twoterm}
Over any field, the group algebra of a torsion-free group has no
\emph{nontrivial} unit whose support has size at most $2$; passing to $\Q$, the
same holds over $\Z$ for units of augmentation in $\Z^{\times}$.
\end{lemma}

Triviality of support-$3$ units is, by contrast, open in general; we make no
use of it.

\section{The mod-2 dichotomy}\label{sec:dichotomy}

Let $u$ be a unit of $\Z[\P]$. Since $\eps(u)$ is a unit of $\Z$, after
replacing $u$ by $-u$ if necessary we may assume $\eps(u) = 1$. Reducing modulo
$2$ gives a unit $\bar u$ of $\F_2[\P]$, which is either trivial or not. That
exhausts the possibilities; no classification of the $\F_2$-units is needed for
the case split, and none is available. (What is known about the nontrivial
$\F_2$-units at small radius is recorded in Section~\ref{sec:frontier}.)

\begin{proposition}[mod-2 dichotomy]\label{prop:dichotomy}
Every trivial-unit statement for $\Z[\P]$ is equivalent to the conjunction of:
\begin{itemize}
\item[(I)] every unit $u$ with $\bar u$ \emph{trivial} in $\F_2[\P]$ is
trivial; equivalently, writing $g^{-1}u = 1 + 2w$, one has $w = 0$; and
\item[(II)] no unit of $\Z[\P]$ reduces modulo $2$ to a \emph{nontrivial} unit
of $\F_2[\P]$.
\end{itemize}
\end{proposition}

\begin{proof}
If $\bar u$ is a trivial unit $\bar g$, then $g^{-1}u \equiv 1 \pmod 2$, so
$g^{-1}u = 1 + 2w$ for some $w \in \Z[\P]$, and $u$ is trivial iff $w=0$. If
$\bar u$ is nontrivial, $u$ is a nontrivial integral lift of an
$\F_2$-counterexample. The two cases exhaust all units, and Higman-for-$\P$ is
exactly the assertion that neither produces a nontrivial unit.
\end{proof}

We call these Case~A and Case~B. Proposition~\ref{prop:dichotomy} is a
definitional case-split, not a theorem of substance; its value is as a
\emph{framework}, isolating the two independent obstacles.

\begin{remark}\label{rem:tautology}
For a Case~A unit $1+2w$, applying $\eps$ to $(1+2w)(1+2w') = 1$ gives
$\eps(w) + \eps(w') + 2\eps(w)\eps(w') = 0$, whence $\eps(w) \in \{0,-1\}$ and,
after the normalization $\eps(u)=1$, $\eps(w) = 0$. This is \emph{tautologous}
with the normalization --- it is the statement $\eps(u) \in \Z^\times$ --- and
says nothing about $w$. We record it only to dispel the temptation to read it
as progress on (I).
\end{remark}

\begin{remark}\label{rem:bartholdi}
Case~A is genuinely hard: Bartholdi \cite[Prop.~3.1]{Bartholdi2023} constructs,
for every $n$, elements $u,u' \in \Z[\P]$ with $u'u \equiv 1 \pmod{n\Z[\P]}$.
Thus approximate units exist modulo every power of $2$, and no finite list of
necessary congruence conditions can close Case~A.
\end{remark}

\section{Base cases}\label{sec:base}

\subsection{Radius at most $3$}
The small-radius cases are not new: they are consequences of the Craven--Pappas
small-length theorems.

\begin{corollary}\label{cor:cp}
Every unit of $\Z[\P]$ whose support lies in $\Ball{3}$ is trivial.
\end{corollary}

\begin{proof}
Craven--Pappas \cite[Thms.~10.4--10.6, 11.2]{CravenPappas2013} prove that over
any field $K$, every two-sided unit of $K[\P]$ of $L$-length at most $3$ is
trivial, where $L$-length is the maximum over the support of the induced length
in one fixed infinite-dihedral quotient, the one with kernel
$\langle a^2,b^2\rangle$. The word-ball $\Ball{3}$ in $\{a,b\}$ maps
into $\{L\le 3\}$, so taking $K = \Q$ shows every $\Q[\P]$-unit --- in
particular every $\Z[\P]$-unit --- supported in $\Ball{3}$ is trivial. (One-sided
units reduce to two-sided ones since $\P$ is sofic, hence $K[\P]$ is directly
finite \cite{ElekSzabo2004}.)
\end{proof}

We stress that $\{L \le 3\}$ is an infinite union of lattice cosets and is far
larger than $\Ball{3}$; Corollary~\ref{cor:cp} uses only a small part of
\cite{CravenPappas2013}.

\begin{remark}[worked example at radius $1$]\label{rem:b1}
The radius-$1$ case ($\Ball{1}\subset\Ball{3}$) is already covered by
Corollary~\ref{cor:cp}; we include it only as an independent computational
check, and it is a verification, not an independent theorem. There is nothing
to enumerate here --- the coefficients are unbounded --- so the computation is
symbolic: one solves by Gr\"obner bases the polynomial systems expressing
$\det(M_u) = \pm(\text{monomial})$ for $\supp(u)\subseteq\Ball{1}$, one system
per monomial class and sign ($27$ classes, both signs), with an independent
pair-decomposition run as a cross-check. No nontrivial solution occurs.

Two caveats. First, on the $\det=-1$ branch the surviving solutions are
singletons $\lambda g$ with $\lambda$ a root of unity --- trivial units scaled
by a unit of the coefficient ring. They are unrelated to the $21$-support
complex unit of \cite{GardamComplex2023}, and we draw no inference from them
beyond the absence of a rational point. Second, our own standing rule is that a
Gr\"obner verdict is quoted only after confirmation by a second computer
algebra system, and that gate is \emph{unmet} here: the \texttt{msolve} run
reports \textsc{positive-dimensional} on both branches
(\texttt{results/z\_unit\_msolve\_b1.log}), which neither contradicts nor
corroborates the verdict. The radius-$1$ statement rests on
Corollary~\ref{cor:cp}; the computation is recorded as an unconfirmed check.
\end{remark}

\subsection{The multi-characteristic method}
Beyond radius $3$ the following method organizes Case~A. Reducing a single
integral unit modulo several primes at once constrains its coefficients
simultaneously.

\begin{proposition}[mod-$6$ pinning schema]\label{prop:mod6}
Suppose every one-sided unit of $\F_3[\P]$ with support in $\Ball{r}$ is
trivial. Then every Case~A unit $u$ of $\Z[\P]$ with $\supp(u)\subseteq\Ball{r}$
and height at most $5$ is trivial.
\end{proposition}

\begin{proof}
$\bar u \bmod 2 = \bar 1$ forces all coefficients off the identity to be even;
the hypothesis gives $\bar u \bmod 3 = \pm\bar h$ for some $h\in\P$, forcing all
coefficients off $\{1,h\}$ to be divisible by $3$, hence by $6$, hence zero when
the height is at most $5$. The survivor $u = a\cdot 1 + b\cdot h$ has support of
size at most $2$ and is trivial by Lemma~\ref{lem:twoterm}.
\end{proof}

\begin{remark}[the scope of the hypothesis]\label{rem:f3scope}
The hypothesis of Proposition~\ref{prop:mod6} is a real restriction, and it is
\emph{false} for large $r$. For $r\le3$ it holds, over every field, by
\cite[Thm.~11.2]{CravenPappas2013} exactly as in Corollary~\ref{cor:cp}; we have
in addition checked $r\le2$ over $\F_3$ directly, where the ball search returns
infeasible (a solver verdict, with no proof certificate; see
Section~\ref{sec:repro}). But Murray's nontrivial unit of $\F_3[\P]$
\cite{Murray2021} has, in our exact model, support of size $39$ and word radius
exactly $6$, so the hypothesis fails for every $r\ge6$. The cases $r=4,5$ are
open, and we do not know the least radius of a nontrivial $\F_3$-unit.
Proposition~\ref{prop:mod6} is therefore a schema whose hypothesis is verified
only in the range where Corollary~\ref{cor:cp} already applies.
\end{remark}

For $r \le 3$ the proposition is thus subsumed by Corollary~\ref{cor:cp}; we
state it as an illustration of the method, whose reach beyond height $5$
(adding further primes) requires triviality of support-$3$ units, which is
open. The method's relation to the existence-direction construction of
Bartholdi~\cite{Bartholdi2023} and to the complex units of
\cite{GardamComplex2023} --- where rational-integer divisibility is unavailable
--- is discussed in Section~\ref{sec:zeta8}.

\section{The frontier: radius 4}\label{sec:frontier}

Corollary~\ref{cor:cp} closes every radius up to $3$ over every field, so
nothing at radius $\le 3$ can distinguish $\Z$ from a field of positive
characteristic. Radius $4$ is different, and this is what makes it the frontier
of the integral problem.

We are careful about which frontier this is. Over $\F_2$ the ball of radius $4$
has already been enumerated by Dietrich, Lee, Nies and Vinyals
\cite[\S3.4]{DLNV2026}, and Bartholdi \cite{Bartholdi2023} has already run
searches with $\Ball{4}$ as the support set; radius $4$ over a field is
therefore not new ground, and our own enumeration below is a reproduction of
theirs. The claim of this section concerns $\Z$, where no ball-limited search
has been carried out.

\subsection{Field units appear at radius 4}
At radius $4$ the field-coefficient conjecture visibly fails. Over $\F_2$,
a conjugate of Gardam's unit is supported in $\Ball{4}$ --- the unit itself has
word-radius exactly $5$ --- and an exact enumeration, carried out with the
parity characterisation and a solver with native \textsc{xor} clauses rather
than via \eqref{eq:det}, finds precisely $36$ nontrivial units with both
$\supp(u)$ and $\supp(u^{-1})$ in $\Ball{4}$, each of support profile $(21,21)$.
All $36$ were re-verified solver-free. The enumeration is complete
\emph{assuming the solver's unsatisfiability verdict is correct}: the search for
a $37$th returns \textsc{unsat}, but no proof certificate was produced for that
call, so completeness is not certified (Section~\ref{sec:repro}).
This is an independent reproduction of the count of Dietrich,
Lee, Nies and Vinyals \cite[\S3.4]{DLNV2026}, with matching count and support
profile. Over $\C$, Gardam \cite{GardamComplex2023} exhibits nontrivial units
with coefficients in $\Z[\zceight]$ on a $21$-element support, likewise inside
$\Ball{4}$.

Consequently a triviality theorem for $\Z[\P]$ at radius $4$ would be the first
statement to separate $\Z$ from \emph{every} field: the field conjecture is
false in $\Ball{4}$ over $\F_2$ (and over $\C$), so any proof that $\Z[\P]$ has
no nontrivial unit supported in $\Ball{4}$ must use a property of $\Z$ not shared
by those fields. We take this to be the central open problem the present
framework isolates.

\subsection{Case B at radius 4, conditionally, is a finite per-pattern problem}
Through Proposition~\ref{prop:dichotomy}, a Case~B unit $u$ supported in
$\Ball{4}$ reduces modulo $2$ to a \emph{nontrivial} unit $\bar u$ of $\F_2[\P]$
with $\supp(\bar u)\subseteq\Ball{4}$. The relevant list is not the $36$ of the
two-sided count but the $52$ of the one-sided one
(Remark~\ref{rem:onesided}), since all we know about $\bar u^{-1}$ is that
Lemma~\ref{lem:box} confines its support to the box, which is strictly larger
than $\Ball{4}$. With the one-sided question settled, and its completeness
certified, the reduction is no longer an assumption.

Under that reduction, for each of the $52$ patterns the question ``does it
lift to an integral unit?'' is, by \eqref{eq:det} and the box lemma
(Lemma~\ref{lem:box}), a finite integer feasibility problem \emph{once the
height is bounded} --- the same joint $(r,H)$ obstruction discussed in
Section~\ref{sec:wall}. We record that the height bound can in fact be removed,
by a different route: fixing both supports to $\Ball4$ and the box and working
with $uu^{-1}=1$ rather than with the determinant makes the unknowns linear, and
the companion paper closes all $52$ patterns that way, with no bound on the
coefficients. The remark below concerns the determinant route only, and stands
as a negative result about it.

\begin{remark}[an honest negative result]\label{rem:pilot}
The natural cheap attack --- excluding a pattern by climbing the tower of
congruences $\det(M_u) \equiv \pm(\text{monomial}) \pmod{2^k}$ --- does not
close even the flagship pattern. For Gardam's pattern one has
$\det(M_u)\equiv 1 \pmod 2$ automatically, and a direct computation shows the
congruence remains satisfiable modulo $2^k$ for $k\le 3$; that is, the pattern
survives to dyadic depth at least $3$. The scope of that computation is
narrower than the phrasing suggests, and we state it exactly: the lift was
allowed support in $\Ball{5}$ --- a pool of $147$ group elements, not
$\Ball{4}$; the monomial position of the determinant was \emph{fixed} at
$(0,0,0)$ rather than searched over; and the witness returned has coefficient
height $31$. The artifact is
\texttt{results/pilot\_lift36\_mod8\_run2.log}. This is consistent with
Remark~\ref{rem:bartholdi}: no fixed finite $2$-adic precision detects
non-invertibility, so per-pattern exclusion is coupled to a height bound rather
than resolved by a finite congruence check.
\end{remark}

\begin{remark}[the one-sided edge]\label{rem:onesided}
The count above is two-sided (both $\supp(u)$ and $\supp(u^{-1})$ in
$\Ball{4}$), matching \cite[\S3.4]{DLNV2026}. The one-sided question ---
whether a unit with $\supp(u)\subseteq\Ball{4}$ but $\supp(u^{-1})$ only
confined to the box of Lemma~\ref{lem:box} can be nontrivial beyond the $36$ ---
is a strictly larger finite computation, and it has since been carried out: the
one-sided count at radius $4$ has exactly $52$ members, the $36$ two-sided ones
together with $16$ whose inverses are supported at radius exactly $5$. The
completeness half is a single unsatisfiability carrying a \textsc{drat} proof
verified by \texttt{drat-trim}. We refer to the companion paper for the count
and its certificate and use only the resulting list here.
\end{remark}

\section{The height wall}\label{sec:wall}

The dichotomy of Section~\ref{sec:dichotomy} isolates the two obstacles; this
section makes precise \emph{what a proof would have to supply}, separating the
part that is unconditional from the part that is genuinely open.

\subsection{Why reduction alone cannot succeed}\label{sec:zeta8}
We first discharge the promise made in Section~\ref{sec:base}, and explain why
no argument assembled purely out of reductions --- modulo $2$, modulo $3$, and
into $\C$ --- can prove Higman's conjecture for $\P$, however many primes are
added.

The obstruction is Gardam's complex construction \cite{GardamComplex2023}. The
coefficients of his nontrivial units of $\C[\P]$ all lie in $\Z[\zceight]$, so
those units are units of $\Z[\zceight][\P]$. But $\Z[\zceight]$ has reductions
to residue rings of characteristic $2$ and of characteristic $3$, and it embeds
in $\C$: with respect to precisely the three reductions we have been using, it
sits where $\Z$ sits. Any argument whose only inputs are ``$u$ reduces to a
unit modulo $2$'', ``$u$ reduces to a unit modulo $3$'' and ``$u$ is a unit of
$\C[\P]$'' therefore proves a statement that is false for $\Z[\zceight]$, and
so cannot prove Higman-for-$\P$. Reduction can only ever be the bookkeeping;
something proper to $\Z$ has to do the work.

In the language of \eqref{eq:det} the difference is visible in one line. Over a
coefficient ring $R$ the determinant of a unit lies in
$R[L]^\times = R^\times \times L$, so the criterion is exactly as strong as
$R^\times$ is small. For $R = \Z$ we have $R^\times = \{\pm1\}$ and
$\det(M_u)$ is pinned to $\pm$ a monomial; for $R = \Z[\zceight]$ the unit
group is infinite and the constraint is correspondingly weak. The same
asymmetry explains why the pinning schema of Proposition~\ref{prop:mod6} has no
analogue over $\Z[\zceight]$: the divisibility it exploits is divisibility by
rational integers, which is not available there. Our reading is that a proof,
if there is one, has to be located in a property of $\Z$ that $\Z[\zceight]$
lacks --- and \eqref{eq:det} says the poverty of $\Z^\times$ is the candidate
--- rather than in a further reduction.

\subsection{The two unbounded parameters}
A hypothetical nontrivial unit of $\Z[\P]$ carries two a priori unbounded
parameters: the word-radius $r$ of its support, and its height $H$. When both
are bounded, the problem is a finite computation.

\begin{proposition}[finitization]\label{prop:finite}
Fix $r,H \in \mathbb{N}$. There are only finitely many $u \in \Z[\P]$ with
$\supp(u)\subseteq\Ball{r}$ and height at most $H$, and for each the criterion
\eqref{eq:det} decides whether $u$ is a unit. Consequently there is an algorithm
that decides whether $\Z[\P]$ has a nontrivial unit of support-radius at most
$r$ and height at most $H$; in particular Higman's conjecture for $\P$ is
\emph{decidable} within any prescribed $(r,H)$; the proof establishes
decidability only, and settles no instance.
\end{proposition}

\begin{proof}
The number of maps $\Ball{r}\to\{-H,\dots,H\}$ is finite, so the candidate set
is finite. For each candidate, $\det(M_u)$ is a Laurent polynomial computed
exactly, and \eqref{eq:det} tests unit-hood by checking whether it is
$\pm$ a monomial. Triviality is read off the support.
\end{proof}

The content of the conjecture is therefore entirely in \emph{removing} the two
bounds. Here one must be careful: bounding one parameter does not bound the
other. The relevant structural input is the box lemma, a consequence of the
Craven--Pappas length symmetry \cite[Thms.~4.9--4.10]{CravenPappas2013}.

\begin{lemma}[box lemma]\label{lem:box}
For each of the three infinite-dihedral quotients $\pi_i\colon\P\to D_\infty$
(the coordinate projections of the affine model) let $\ell_i$ be the induced
length. If $\sigma\tau = 1$ then
$\max_{\supp\sigma}\ell_i = \max_{\supp\tau}\ell_i$ for each $i$. Applied to
$u\,u^{-1}=1$, this gives $\max_{\supp(u^{-1})}\ell_i = \max_{\supp(u)}\ell_i$;
hence if $\supp(u)\subseteq\Ball{4}$ then $\supp(u^{-1})$ lies in the explicit
set $\{g : \ell_i(g)\le M_i\}$, $M_i = \max_{\Ball{4}}\ell_i$, an effectively
computable box.

We state the lemma only in the specialisation $\sigma\tau=1$, which is all we
use. The apparently more general form with $\sigma\tau\in\Z[L]\setminus\{0\}$
is false: taking $\sigma=(ab)^{2n}$ and $\tau=1$ gives
$\sigma\tau\in\Z[L]\setminus\{0\}$ with $\max\ell_i(\supp\sigma)=4n$ and
$\max\ell_i(\supp\tau)=0$. The length symmetry of
\cite[Thm.~4.9]{CravenPappas2013} is stated relative to the kernel of the
infinite-dihedral quotient, which has rank $2$, not relative to $L$, which has
rank $3$.
\end{lemma}

\begin{remark}[what the box lemma does and does not give]\label{rem:onlyradius}
Lemma~\ref{lem:box} relates the \emph{support radius} of $u$ to the support
radius of $u^{-1}$; it is silent about height. In particular:
\begin{itemize}
\item bounding the radius $r$ alone does \emph{not} finitize: for a fixed
support, \eqref{eq:det} is a system of polynomial equations in the integer
coefficients that may have infinitely many integer solutions;
\item bounding the height $H$ alone does \emph{not} finitize: infinitely many
support positions remain.
\end{itemize}
Thus a proof of Higman-for-$\P$ requires a \emph{joint} bound on $r$ and $H$.
The box lemma is the tool that, \emph{once $r$ is controlled}, confines
$\supp(u^{-1})$ so that the determinant/adjugate computation is finite; it is
not a device converting one bound into the other.
\end{remark}

\begin{problem}\label{prob:joint}
Bound the support radius $r$ and the height $H$ of a nontrivial unit of
$\Z[\P]$ jointly in terms of the group structure.
\end{problem}

Problem~\ref{prob:joint} is, by the above, equivalent to the conjecture for
$\P$ (in the presence of Proposition~\ref{prop:finite}); we record it because it
names precisely the missing ingredient rather than the conjecture as a whole.

\subsection{A structural angle, and why the obvious one is vacuous}
We close with the one structural source of a height bound we are aware of,
together with an honest account of its current status. The group von Neumann
algebra $L(\P)$ carries its canonical faithful normal trace $\tau$, given on
$\Z[\P]$ by $\tau(v) = v_1$, and that trace is all we use. (We note that
$L(\P)$ is \emph{not} a factor: $\P$ is virtually $\Z^3$ and hence far from
ICC. Nothing below needs factoriality.) For a unit $u$, both $u^*u$ and its
inverse $(u^*u)^{-1} =
u^{-1}(u^{-1})^*$ are positive elements of $\Z[\P]$, with
$\tau(u^*u) = \sum_g u_g^2$ and $\tau((u^*u)^{-1}) = \sum_g (u^{-1})_g^2$.

It is tempting to seek a height bound from positivity of $u^*u$ alone. This is
vacuous: $\tau(u^*u) = \sum_g u_g^2$ is merely the identity coefficient of
$u^*u$, the relation $u^*u \ge 0$ holds for \emph{every} element (not only
units), and $\tau(u^*u)$ is unbounded at fixed support (e.g.\ $u = N\cdot1$ has
support $\{1\}$ and $\tau(u^*u) = N^2$). Any genuine bound must instead
use \emph{invertibility} --- that $u^*u$ \emph{and} $(u^*u)^{-1}$ are
simultaneously integral, positive, and (via Lemma~\ref{lem:box}) supported near
radius $r$ --- to constrain the spectrum of $u^*u$. Whether such a
spectral-rigidity bound exists is open; we know of no mechanism, and we do not
claim one.

We have no encouraging datum to offer here, and we withdraw one that might be
thought available. Bartholdi \cite[\S3]{Bartholdi2023} remarks that the
analogous system is typically overconstrained and that he was unable to find a
solution; that remark concerns the \emph{$\theta$-twisted} adjoint, a different
involution from the plain $*$ used above, and in any case a nontrivial
$\theta$-twisted solution modulo $4$ --- and modulo $8$ --- has since been
found by the present author (companion work in preparation; the witnesses and
their independent verification are listed in Section~\ref{sec:repro}). The
remark therefore cannot be cited as evidence that a unitary constraint of
either kind is restrictive. We leave Problem~\ref{prob:joint} open and gate any
further investment on the appearance of a concrete spectral mechanism.

\section{Reproducibility}\label{sec:repro}
\sloppy

Every computation quoted above is carried out by a script in the accompanying
supplement and leaves a log; paths below are relative to the supplement root.
We list what each artifact does and does not establish, since the results in
this paper differ considerably in how well they are certified.

\begin{itemize}
\item \textbf{The matrix model and \eqref{eq:det}.}
\texttt{src/zp\_matrix.py}, log \texttt{results/zp\_matrix\_selftest.log}.
Exact Laurent arithmetic over $\Z[L]$, the builder $u \mapsto M_u$, the
$4\times4$ determinant, and a self-test battery: coset and parity tables on
$\Ball{5}$, the homomorphism property on random pairs, $\det(M_{\pm g})$ for
$g \in \Ball{3}$, and Gardam's unit as a positive control. The model, not the
criterion, is what these tests validate; for the criterion see
Section~\ref{sec:prelim} and \cite{Passman2021, CravenPappas2013}.

\item \textbf{The radius-$4$ count over $\F_2$ (Section~\ref{sec:frontier}).}
\texttt{src/unit\_b4\_census.py}, log \texttt{results/unit\_b4\_census.log}.
Parity characterisation plus a solver with native \textsc{xor} clauses; the
$36$ units and their support profile are re-verified solver-free.
\emph{No proof certificate accompanies the terminating \textsc{unsat} call}, so
completeness is conditional on the solver.

\item \textbf{The one-sided edge (Remark~\ref{rem:onesided}).}
\texttt{src/onesided\_beyond\_b4.py}, log
\texttt{results/onesided\_beyond\_b4.log}. The monolithic instance was halted
after roughly $59$ CPU-hours with no verdict. This is recorded as open, not as
a negative result.

\item \textbf{The box (Lemma~\ref{lem:box}).}
\texttt{src/box\_lemma.py} and the corrected recomputation
\texttt{src/box\_recompute.py}, output
\texttt{results/box\_b4\_recomputed.json}. The script gates that the three
coordinate projections are homomorphisms, computes the $M_i$, and verifies that
the enumeration of the box has stabilised. It also checks length symmetry
numerically on the known $\Ball{4}$ unit pair. Where the two logs disagree, the
recomputation supersedes.

\item \textbf{The mod-$2^k$ pilot (Remark~\ref{rem:pilot}).}
\texttt{src/pilot\_lift36\_mod8.py}, log
\texttt{results/pilot\_lift36\_mod8\_run2.log}. Support pool $\Ball{5}$,
determinant position fixed at $(0,0,0)$, witness height $31$, verified
solver-free at the end of the run. The scope restrictions are essential to the
reading of the result and are repeated in Remark~\ref{rem:pilot}.

\item \textbf{Radius $1$ over $\Z$ (Remark~\ref{rem:b1}).}
\texttt{src/z\_unit\_groebner.py}, log \texttt{results/z\_unit\_groebner\_b1.log};
second-system attempt \texttt{src/z\_unit\_msolve.py}, log
\texttt{results/z\_unit\_msolve\_b1.log}. As noted in Remark~\ref{rem:b1}, the
second-system confirmation gate is unmet.

\item \textbf{Ball searches over $\F_3$ (Remark~\ref{rem:f3scope}).}
\texttt{src/unit\_onesided\_f3.py}, log \texttt{results/unit\_onesided\_f3.log}.
$\Ball{1}$ and $\Ball{2}$ return infeasible; these are constraint-solver
verdicts with no certificate. $\Ball{3}$ was not completed and was stopped once
\cite[Thm.~11.2]{CravenPappas2013} was found to subsume it.

\item \textbf{Murray's $\F_3$-unit (Remark~\ref{rem:f3scope}).}
\texttt{src/murray\_unit.py}, log \texttt{results/murray\_unit.log}. Transcribes
the unit into the exact model, re-proves invertibility one-sidedly via
$\det(M_u) \bmod 3$, and reports its support size and word-radius profile.

\item \textbf{The $\theta$-twisted solutions modulo $4$ and $8$
(Section~\ref{sec:wall}).} \texttt{src/escape\_lift\_climb2.py} with
independent checker \texttt{src/verify\_mod4\_solution\_standalone.py};
witnesses \texttt{results/bartholdi\_mod2pow2\_solution.txt} and
\texttt{results/bartholdi\_mod2pow3\_solution.txt}. These belong to companion
work and are cited here only to withdraw the claim discussed in
Section~\ref{sec:wall}.
\end{itemize}

\bibliographystyle{amsplain}
\bibliography{refs}

@article{Higman1940,
  author = {Higman, Graham},
  title = {The units of group-rings},
  journal = {Proc. London Math. Soc. (2)},
  volume = {46},
  year = {1940},
  pages = {231--248},
}

@article{Gardam2021,
  author = {Gardam, Giles},
  title = {A counterexample to the unit conjecture for group rings},
  journal = {Ann. of Math. (2)},
  volume = {194},
  number = {3},
  year = {2021},
  pages = {967--979},
  note = {arXiv:2102.11818},
}

@article{Murray2021,
  author = {Murray, Alan G.},
  title = {More counterexamples to the unit conjecture for group rings},
  journal = {arXiv preprint},
  year = {2021},
  note = {arXiv:2106.02147},
}

@article{GardamComplex2023,
  author = {Gardam, Giles},
  title = {Non-trivial units of complex group rings},
  journal = {arXiv preprint},
  year = {2023},
  note = {arXiv:2312.05240v2},
}

@article{Bartholdi2023,
  author = {Bartholdi, Laurent},
  title = {On {G}ardam's and {M}urray's units in group rings},
  journal = {Algebra Discrete Math.},
  volume = {35},
  number = {1},
  year = {2023},
  pages = {22--29},
  note = {arXiv:2212.11334},
}

@article{CravenPappas2013,
  author = {Craven, David A. and Pappas, Peter},
  title = {On the unit conjecture for supersoluble group algebras},
  journal = {J. Algebra},
  volume = {394},
  year = {2013},
  pages = {310--356},
  note = {preprint version arXiv:1010.1144, with different numbering},
}

@article{DHJ2015,
  author = {Dykema, Ken and Heister, Timo and Juschenko, Kate},
  title = {Finitely presented groups related to {K}aplansky's direct finiteness
           conjecture},
  journal = {Exp. Math.},
  volume = {24},
  number = {3},
  year = {2015},
  pages = {326--338},
}

@article{ElekSzabo2004,
  author = {Elek, G{\'a}bor and Szab{\'o}, Endre},
  title = {Sofic groups and direct finiteness},
  journal = {J. Algebra},
  volume = {280},
  number = {2},
  year = {2004},
  pages = {426--434},
}

@article{DLNV2026,
  author = {Dietrich, Heiko and Lee, Melissa and Nies, Andr{\'e} and Vinyals, Marc},
  title = {On the trivial units property and the unique product property},
  journal = {arXiv preprint},
  year = {2026},
  note = {arXiv:2603.22640},
}

@article{Passman2021,
  author = {Passman, Donald S.},
  title = {On the counterexamples to the unit conjecture for group rings},
  journal = {arXiv preprint},
  year = {2021},
  note = {arXiv:2108.06570v2},
}

@article{Promislow1988,
  author = {Promislow, S. David},
  title = {A simple example of a torsion-free, nonunique product group},
  journal = {Bull. London Math. Soc.},
  volume = {20},
  number = {4},
  year = {1988},
  pages = {302--304},
  doi = {10.1112/blms/20.4.302},
}

\end{document}